\documentclass[12pt,reqno]{amsart}
\usepackage{amsmath,amssymb}

\def\P{\mathbb P}%
\def\C{\mathbb C}%
\def\S{\mathbb S}%
\def\cP{\mathcal P}%
\def\cQ{\mathcal Q}%
\def\cD{\mathcal D}%
\def\F{\mathcal F}%
\def\cO{\mathcal O}%
\def\x{\bar x}
\def\tY{\tilde Y}%
\def\ty{\tilde y}%
\def\y{\bar y}
\def\q{\bar q}
\def\tq{\tilde q}
\def\bP{\bold P}
\def\bq{\bold q}
\def\bQ{\bold Q}
\def\Sing{\operatorname{Sing}}%
\def\Supp{\operatorname{Supp}}%
\def\Sm{\operatorname{Sm}}%
\def\Ker{\operatorname{Ker}}%
\def\codim{\operatorname{codim}}%
\def\id{\operatorname{Id}}%
\def\rk{\operatorname{rk}}%
\def\cork{\operatorname{cork}}%
\def\<{\langle}%
\def\>{\rangle}%
\def\Def{\operatorname{def}}%
\def\pDef{\operatorname{pdef}}%
\def\Hom{\operatorname{Hom}}%
\def\Sym{\operatorname{Sym}}%
\def\GL{\operatorname{GL}}%
\def\gl{\operatorname{\mathfrak gl}}%
\def\g{\mathfrak g}
\theoremstyle{plain}
\newtheorem{thm}{Theorem}[section]
\newtheorem{theorem}[thm]{Theorem}
\newtheorem{lemma}[thm]{Lemma}
\newtheorem{corollary}[thm]{Corollary}
\newtheorem{proposition}[thm]{Proposition}
\theoremstyle{definition}

\newtheorem{remark}[thm]{Remark}
\newtheorem{definition}[thm]{Definition}
\newtheorem{example}[thm]{Example}

\newtheorem*{question}{Question}
\newtheorem*{main}{Main Theorem}
\numberwithin{equation}{section}
\def\aut{\operatorname{\mathfrak aut}}%
\def\autp{\aut^1}%

\begin{document}

\title[A characterization of secant varieties of Severi varieties]
{A characterization of secant varieties of Severi varieties among
cubic hypersurfaces}

\author{Baohua Fu}\thanks{Baohua Fu is supported by
National Natural Science Foundation of China (Nos. 11771425 and 11688101).}
\author{Yewon Jeong}\thanks{Yewon Jeong is supported by
Postdoctoral International Exchange Program (Y890172G21).}
\author{Fyodor L. Zak} \maketitle

\begin{abstract}
It is shown that an irreducible cubic hypersurface with nonzero
Hessian and smooth singular locus is the secant variety of a Severi
variety if and only if its Lie algebra of infinitesimal linear
automorphisms admits a nonzero prolongation.
\end{abstract}

\section{Introduction}

Let $V^m\subset \P^N$ be an $m$-dimensional irreducible nondegenerate
smooth complex projective variety.  The secant variety $SV$ is the
closure of the union of lines in $\P^N$ joining two distinct points of
$V$. It is easy to see that $V$ can be isomorphically projected to a
lower-dimensional projective space if and only if $SV\neq\P^N$, which
can only occur if $m$ is not too big. As conjectured by Hartshorne
and proved by Zak (\cite{Z1}), $SV=\P^N$ provided that $m
>\frac{2N-4}{3}$. We call $V$  a {\em Severi variety} if
$SV\neq\P^N$ and $m =\frac{2N-4}{3}$.

As proved by Zak (\cite{Z1}), there exist exactly four Severi
varieties:
$$
v_2(\P^2)\subset \P^5,\quad\P^2\times\P^2\subset\P^8,
\quad G(1,5)\subset\P^{14},\quad\mathbb{OP}^2\subset\P^{26},
$$
viz. the Veronese surface ($m=2$), the Segre variety ($m=4$), the
Grassmann variety of lines in $\P^5$ ($m=8$), and the Cayley plane
corresponding to the closed orbit of the minimal representation of
the algebraic group $E_6$ ($m=16$). The vector spaces corresponding
to the ambient spaces $\P^N$ of Severi varieties can be identified
with the spaces of Hermitian $3\times3$ matrices with coefficients in
composition algebras, and under this identification the affine cones
corresponding to Severi varieties are the loci of matrices whose rank
does not exceed one. The secant varieties of Severi varieties are
irreducible cubic hypersurfaces defined by vanishing of the
determinant of the corresponding $3\times3$ matrix, and the
projective duals of these cubics are naturally isomorphic to the
corresponding Severi varieties.

Severi varieties form the third row of the so-called Freudenthal
magic square, and their rich projective geometry was thoroughly
studied (see e.g. \cite{Z1}, \cite{LM} and \cite{IM}). They are also
related to homogeneous Fano contact manifolds since the latter can be
recovered from Severi varieties. Thus a better understanding of
Severi varieties can shed some light on the long-standing conjecture
of LeBrun and Salamon predicting that all Fano contact manifolds are
homogeneous. This motivates the problem of characterizing the secant
varieties of Severi varieties among all cubic hypersurfaces.
Following \cite{H}, we solve this problem in terms of prolongations
of infinitesimal automorphisms.

Let $W$ be a complex vector space. A prolongation of a Lie subalgebra
$\g\subset\gl{W}$ is an element $A\in\Hom{(\Sym^2 W,W)}$ such that
$A(w,\cdot)\in\g$ for any $w \in W$. The vector space of all
prolongations of $\g$ is denoted by $\g^1$. Let $T\subset\P W$ be a
smooth projective variety, let $\hat T\subset W$ be the corresponding
affine cone, and let $\aut{\hat{T}}\subset\gl{W}$ be the Lie algebra
of infinitesimal linear automorphisms of $\hat T$. We are interested
in the vector space $\autp{\hat T}$ of all prolongations of $\aut{\hat
T}$.

Let $Y \subset \P W$ be an irreducible cubic hypersurface defined by
a cubic form $F\in\Sym^3{W^*}$, and let $\hat Y$ be its affine cone.
Both $\aut{\hat Y}$ and $\autp{\hat Y}$ can be computed effectively
in terms of $F$ (cf. Section~\ref{s.3}). In the case when $Y=SV$ for
a Severi variety $V$ one has $\autp{\hat Y}\neq0$ (cf.
Corollary~\ref{c.pSeveri}). For various reasons it makes sense to
focus the study of prolongations on the case when the {\it polar
map\/} defined by the partial derivatives of $F$ is surjective or,
equivalently, the Hessian determinant of $F$ is not (identically)
equal to zero. In this case we say that $Y$ is not {\it polar
defective\/} (cf. Definition~\ref{d.pd} in the next section).

J.-M. Hwang posed the following question (Question~1.3 in \cite{H}).

\begin{question}
{\it Let $\,Y$ be an irreducible cubic hypersurface. Is it true that
if $\autp{\hat Y} \neq 0$ and $\,Y$ is not polar defective, then
$\,Y$ is the secant variety of a Severi variety?}
\end{question}

It turns out that in general the answer to this question is negative;
cf. e.g. Example~\ref{e.SectionSecSeveri}\,(ii). In the present paper
we give a positive answer to Hwang's question under the additional
assumption that the (reduced) singular locus $Y'\subset Y$ is smooth.

\begin{main}
{\it Let $Y\subset\P W$ be an irreducible cubic hypersurface.  Assume
that
\begin{itemize}
\item[a)] $Y$ is not polar defective{\rm;}
\item[b)] $Y'$ is smooth{\rm;}
\item[c)] $\autp{\hat Y}\neq 0$.
\end{itemize}
Then $Y$ is the secant variety of a Severi variety.}
\end{main}

It should be mentioned that in \cite{H} J.-M.~Hwang proved a weaker
version of this result in which assumption~(c) is replaced by
\begin{itemize}
\item[$\text c')$] $\Xi^{\,a}_{\,Y}\ne0$ {\it for some} $a\ne\frac14$,
\end{itemize}
where $a\in\C$ is a complex number and $\Xi^{\,a}_{\,Y}\subset
\autp{\hat Y}$ is a certain linear subspace with a rather intricate
definition (cf. Theorem~1.6 in \cite{H}), thus giving a partial
answer to Question~1.5 in \cite{H} which is a weaker form of the
above Question.

As suggested in \cite{H}, the proof of the Main Theorem splits into
two parts: the first one is to show that $Y =SY'_0$ for an
irreducible component $Y'_0\subset Y'$ and the second one is to go
through the classification of smooth nondegenerate projective
varieties with nonzero prolongation given in \cite{FH1} and
\cite{FH2}. In this paper we mainly contribute to the first part of
this strategy by exploring the relationship between dual and polar
defectivity (the latter is equivalent to the classical notion of
vanishing Hessian).

It is easy to see that any irreducible hypersurface with vanishing
Hessian is dual defective, but the converse is not true, as is shown
by the secant varieties of Severi varieties. In
Theorem~\ref{t.PolarDefective} we show that if $Y$ is an irreducible
dual defective cubic with smooth $Y'$ such that  $SY'\subsetneq Y$, then the defining
equation $F$ of $Y$ has vanishing Hessian. At the second step of the
proof of the Main Theorem we anyhow need to assume that $Y'$ is
smooth, and so the smoothness assumption in
Theorem~\ref{t.PolarDefective} is not restrictive. In the meantime
the third named author proved that Theorem~\ref{t.PolarDefective} is
true even without this assumption.

At the end of the paper we give examples showing that none of the
conditions~a)--c) of the Main Theorem can be lifted.

\subsection*{Acknowledgement}
The authors are grateful to Jun-Muk Hwang for helpful discussions and
suggestions and to the anonymous referees for carefully reading the
manuscript and making useful comments that helped us to straighten
the exposition.

\section{Dual defective cubic hypersurfaces}

Let $Y\subset\P^N$ be an irreducible complex projective hypersurface
defined by a homogeneous polynomial $F$ of degree $d>1$. Let
$\phi:\P^N\dasharrow \P^{N\,{}^*}$ be the {\it polar map\/} given by
the {\it polar linear system\/} $\cQ$ whose members are cut out by
the partial derivatives of $F$, let $X^n=Y^*\subset\P^{N\,{}^*}$ be
the dual variety, and let $\gamma=\phi_{|Y}$ be the {\it Gauss map}
(we refer to \cite{K} and \cite{T} for basic facts on dual varieties).
Both $\phi$ and $\gamma$ are defined outside of the singular subset
$Y'=(\Sing{Y})_{\text{red}}$ which is also the base locus of $\cQ$.

\begin{definition} The integer $\Def{Y}=\codim{X}-1=N-n-1$ is called
the {\it dual defect\/} of $Y$. The hypersurface $Y$ is called {\it
dual defective\/} if $\Def{Y}>0$, i.e. $\!X$ fails to be a
hypersurface.
\end{definition}

It is well known that, for a general point $x\in X$, the fiber
$\gamma^{-1}(x)={}^{\perp}T_{X,x}\subset Y$ is a linear subspace of
dimension $\Def{Y}$ (in particular, if $\Def{Y}=0$, then the map
$\gamma:Y\dasharrow X$ is birational).

\begin{definition}\label{d.pd} An irreducible hypersurface $Y$ is called
{\it polar defective\/} if it satisfies one of the following
equivalent conditions:
\begin{itemize}
\item[a)] $\phi(\P^N)=Z\subsetneq\P^{N\,{}^*}$;
\item[b)] $F$ has vanishing Hessian, i.e. $\det{H}\equiv0$, where
$H$ is the Hesse matrix formed by the second order partial
derivatives of $F$.
\end{itemize}
The number $\pDef{Y}=\codim{Z}=N-r$, where $r=\dim{Z}$, is called the
{\it polar defect\/} of $Y$, so that $Y$ is polar defective if and
only if $\pDef{Y}>0$. It is easy to see that $\dim{Z}=\rk{H}-1$ and
$\pDef{Y}=\codim{Z}=\cork{H}$, where $\cork{H}=N+1-\rk{H}$ is the
corank of $H$.
\end{definition}

From the proof of \cite[Proposition~4.9\,(ii)]{Z2} it follows that,
for a general point $z\in Z$, the fiber $\F_z=\phi^{-1}(z)
\subset\P^N$ is a union of finitely many linear subspaces of
dimension $\pDef{Y}$ passing through the linear subspace
${}^{\perp}T_{Z,z}\subset Z^*$ of dimension $\pDef{Y}-1$.
Furthermore, $Z^*\subset Y'$ and $\F_z\cap Y= \F_z\cap Y'=\F_z\cap
Z^*={}^{\perp}T_{Z,z}$.

\smallskip

The simplest example of polar defective hypersurface is given by
cones (in which case both $Z$ and $X$ are degenerate varieties), but
there exist many more interesting examples (cf. e.g.
Example~\ref{e.SectionSecSeveri}\,(iii)).

\begin{proposition}\label{p.PD}
\hfill
\begin{itemize}
\item[(i)] $Z\supsetneq X$, i.e. $n+1\le r\le N$;
\item[(ii)] Any polar defective hypersurface is dual defective.
More precisely, $\Def{Y}\ge\pDef{Y}$ and the inequality is strict if
and only if $r>n+1$.
\end{itemize}
\end{proposition}

\begin{proof} (i) It is clear that $X=\gamma(Y)=\phi(Y)\subset
\phi(\P^N)=Z$. Thus we only need to show that $Z\ne X$. Suppose to
the contrary that $Z=X$, and let $x\in X$ be a general point. Then
$\gamma^{-1}(x)={}^{\perp}T_{X,x}$ is a linear subspace of dimension
$N-n-1$ contained in the $(N-n)$-dimensional fiber $\F_x=
\phi^{-1}(x)\subset\P^N$. Furthermore, since $\gamma=\phi_{|Y}$,
$\F_x\cdot Y=\gamma^{-1}(x)=\P^{N-n-1}$, which is only possible if
$\F_x$ is a linear subspace and $Y$ is a hyperplane, contrary to the
hypothesis that $d>1$. \qed

(ii) is an immediate consequence of (i).
\end{proof}

\begin{remark}\label{r.PD} The converse of Proposition~\ref{p.PD}\,(ii) is
false. For example, let $V_i^{n_i}\subset\P^{N_i}$, $1\le i\le4$,
$n_i=2^i$, $N_i=\frac{3n_i}2+2$ be the $i$-th Severi variety, and let
$Y_i=SV_i\subset\P^{N_i}$ be its secant variety. Then $Y_i$ is a
cubic hypersurface singular along $Y_i'=V_i$, $X_i=\gamma_i(Y_i)
\subset\P^{N_i\,{}^*}$ is also the $i$-th Severi variety, and
$\phi_i:\P^{N_i}\dasharrow\P^{N_{i}\,{}^*}$ is the birational Cremona
transformation contracting $Y_i$ to $X_i$, blowing up $V_i$ to the
cubic $SX_i$, and defining an isomorphism between the complements of
the cubic hypersurfaces in $\P^N$ and $\P^{N\,{}^*}$ (cf. \cite[Chap.
IV]{Z1}). In particular, $\Def{Y_i}=\frac{n_i}2+1$, but
$\pDef{Y_i}=0$ and $Y_i$ is not polar defective.
\end{remark}

\begin{proposition}\label{p.HyperplaneSection} Let $L\subset\P^N$ be
a general hyperplane, and let $Y_L=L\cap Y$ be the corresponding
hyperplane section of $Y$. Then either $\pDef{Y}=\pDef{Y_L}=0$ or
$\pDef{Y}=\pDef{Y_L}+1$. In particular, if $Y_L$ is polar
defective, then so is $Y$.
\end{proposition}

\begin{proof} Let $\phi:\P^N\dasharrow\P^{N\,{}^*}$ and $\phi_L: L
\dasharrow L^*$ be the polar maps corresponding to $Y$ and $Y_L$
respectively, let $Z\subset\P^{N\,{}^*}$ and $Z_L \subset L^*$ be
their respective images, and let $\pi:\P^{N\,{}^*}\dasharrow L^*$ be
the projection with center at the point ${}^{\perp}L$ corresponding
to the hyperplane $L$. Then
\begin{equation}\label{p.dp}
\dim{\phi(L)}=
\begin{cases}\dim{Z}=N-\pDef{Y}&\text{if}\ \ \pDef{Y}>0,\\
\dim{Z}-1=N-1&\text{if}\ \ \pDef{Y}=0.
\end{cases}
\end{equation}

We use the following

\begin{lemma}\label{l.L}
Let $L\subset\P^N$ be a general hyperplane, and let ${}^{\perp}L$ be
the corresponding point in $\P^{N\,{}^*}$. Then ${}^{\perp}L\notin
\phi(L)$.
\end{lemma}

\begin{proof}
If $\phi$ fails to be dominant, we can just take any $L$ for which
${}^{\perp}L\notin Z$. Suppose now that $\phi$ is dominant. Let
$\Gamma\subset\P^N\times\P^{N\,{}^*}$ denote the closure of the graph
of $\phi$, and let $\Gamma\xrightarrow p\P^N$ and $\Gamma\xrightarrow
q\P^{N\,{}^*}$ be the natural projections. Let $L$ be a hyperplane
for which ${}^{\perp}L\notin D\cup X$, where $D=
q(p^{-1}(Y'))\subsetneq\P^{N\,{}^*}$. Then ${}^{\perp}L\notin
\phi(L)$ since otherwise ${}^{\perp}L=\phi(v)$ for a point $v\in L
\setminus Y'$ and from the Euler formula it follows that $v\in Y\cap
L$ and ${}^{\perp}L=\phi(v)\in\gamma(Y)=X$, a contradiction.
\end{proof}

We return to the proof of Proposition~\ref{p.HyperplaneSection}.
Since $Z_L=\phi_L(L)=\pi(\phi(L))$ and, by Lemma~\ref{l.L}, $\phi(L)$
is not a cone with vertex ${}^{\perp}L$, from~\eqref{p.dp} it follows
that
$$
\dim{Z_L}=\dim{\phi(L)}=\begin{cases}
\dim{Z}&\text{if}\ \ \pDef{Y}>0\\
\dim{Z}-1&\text{if}\ \ \pDef{Y}=0
\end{cases},
$$
and so
$$
\pDef{Y_L}=N-\dim{Z_L}-1=
\begin{cases} N-\dim{Z}-1=\pDef{Y}-1&\text{if}\ \ \pDef{Y}>0,\\
N-(\dim{Z}-1)-1=\pDef{Y}&\text{if}\ \ \pDef{Y}=0.
\end{cases}
$$
\end{proof}

For an arbitrary (irreducible) hypersurface $Y\subset\P^N$, consider
the {\it conormal variety\/} $\cP\subset X\times Y$, $\cP=
\overline{\{(x,y)\mid x\in\Sm{X},{}^{\perp}y\supset T_{X,x}\}}$,
where $X^n=Y^*$, $\Sm{X}=X\setminus\Sing{X}$ is the open subset of
nonsingular points of $X$ and $T_{X,x}$ is the embedded tangent space
to $X$ at $x$, and let $p:\cP\to X$ and $\pi:\cP\to Y$ denote the
projections. Then $\pi$ is birational, $\gamma=p\circ\pi^{-1}$,
$p_{|\Sm{X}}$ is a $\P^{N-n-1}$ bundle and, for $x\in\Sm{X}$,
$\pi(\cP_x)=\gamma^{-1}(x)$. For a general $x\in X$, consider the
locus $\cP'_x$ of points $y\in\cP_x$ for which the hyperplane section
${}^{\perp}y\cdot X$ fails to have a nondegenerate quadratic
singularity at $x$. An easy computation (cf. e.g.
\cite[Proposition~3.3]{K}) shows that $\pi$ is ramified at $(x,y)$
iff $y\in\cP'_x$, whence $\cP'_x\subset Y'\cap\cP_x$. Let $\cP''_x
\subset\cP'_x\subset\cP_x$ be the linear subspace of hyperplanes that
are tangent to $X$ at $x$ to order larger than $2$ (i.e. inducing an
element of $\mathfrak m_x^3$, where $\mathfrak m_x\subset \cO_x$ is
the maximal ideal of $x$), let $\cQ_n=\P^{(n-1)(n+2)}$ be the
projective space of quadratic forms in $n$ variables (or quadrics in
$\P^{n-1}$), let $\cQ'_n\subset\cQ_n$ be the subvariety of degenerate
quadratic forms (or singular quadrics), and let $\tau:\cP_x
\dasharrow\cQ_n$ denote the projection with center at $\cP''_x$,
which sends a point $y \in \cP_x \setminus \cP''_x$ to the initial
quadratic form of the Taylor expansion near $x$ of the hyperplane
section ${}^{\perp}y\cdot X$. Then $\cP'_x=\tau^{-1}(\cQ'_n)$, and
since $\cQ'_n\subset\cQ_n$ is a hypersurface of degree $n$,
$\cP'_x\subset\cP_x$ is a hypersurface of degree not exceeding $n$
(cf. also the paragraph after the proof of Corollary~5.7 in
\cite[Chapter IV]{Z1}). On the other hand, $\cP'_x\subset Y'\cap
\cP_x$ and $Y'$ is defined by equations of degree $d-1$, hence
$\deg{\cP'_x}\le d-1$. In particular, if $Y$ is a cubic, then
$\deg{\cP'_x}\le2$. The locus of $\cP'_x$ in $\cP$ will be denoted by
$\cP'$.

\smallskip

For the rest of this section we restrict our attention to the case
when $Y$ is an irreducible cubic hypersurface.

\begin{proposition}\label{p.Sec}
Let $Y \subset \P^N$ be a cubic hypersurface. Then $SY'\subseteq Y$.
\end{proposition}

\begin{proof}
If a line joining two distinct points of $Y'$ were not contained in
$Y$, it would meet $Y$ with multiplicity at least four while
$\deg{Y}=3$.
\end{proof}

\begin{lemma} \label{l.} Suppose that $SY'\subsetneq Y$, and let
$x\in X$ be a general point. Then $\cP'_x$ is a hyperplane in
$\cP_x$.
\end{lemma}

\begin{proof} In fact, if, for a general $x\in X$, $\cP'_x$ were a
quadric, then one would have $S\cP'_x=\cP_x$, hence $SY'=Y$.
\end{proof}

Denote by $\pi'$ the restriction of $\pi$ on $\cP'$ and by $Y''$
the image of $\cP'$ in $Y$. Then $Y''=\pi(\cP')\subset Y'$ is an
irreducible subvariety. For a general point $y\in Y''$, put $X_y=
p((\pi')^{-1}(y))$.

By Proposition~\ref{p.PD}, polar defectivity implies dual
defectivity, and Remark~\ref{r.PD} shows that the converse is not
true even for cubics. However, in the examples in Remark~\ref{r.PD}
one has $SY'_i=Y_i$. This is not accidental: it turns out that a dual
defective cubic hypersurface is polar defective provided that
$SY'\subsetneq Y$. The main goal of this section is to prove this
under the additional assumption that $Y'$ is smooth (cf. however
Remark~\ref{r.sing}).

\begin{theorem}\label{t.PolarDefective} Let $Y\subset\P^N$ be an
irreducible dual defective cubic hypersurface. Suppose that $Y'$ is
smooth and $SY'\subsetneq Y$. Then $Y$ is polar defective.
\end{theorem}

\begin{proof}[Proof of Theorem~\ref{t.PolarDefective}] We split the
proof into a series of lemmas.

\begin{lemma}\label{l.red} It suffices to prove
Theorem~\ref{t.PolarDefective} in the case when $Y$ is not a cone and
$N=n+2$ {\rm(}i.e. $\Def{Y}=1${\rm)}.
\end{lemma}

\begin{proof} Since cones are polar defective, we can
assume that $Y$ is not a cone. Let $L\subset\P^N$, $\dim{L}=n+2$
be a general linear subspace, let $Y_L=L\cap Y$, and put
$X_L=Y_L^*$. From Theorem~1.21 in \cite{T} or Proposition~2.4 in
\cite{Z2} it follows that $X_L$ is obtained from $X$ by
projecting from the (general) linear subspace ${}^{\perp}L\subset
\P^{N\,{}^*}$, hence $\dim{X_L}=\dim{X}=n$ and $\Def{Y}=1$. The
hypotheses that $Y'$ is smooth and $SY'\subsetneq Y$ are clearly
stable with respect to passing to a general linear section.
Therefore Lemma~\ref{l.red} follows from
Proposition~\ref{p.HyperplaneSection}.
\end{proof}

From now on we assume that $N=n+2$ and $Y$ is not a cone. We denote
by $\<A\>$ the linear span of a subset $A\subset\P^N$.

Let $x\in X$ be a general point, and let $l=l_x=\pi(\cP_x)=
\gamma^{-1}(x)$. By Lemma~\ref{l.}, $l\cap Y'=l\cap Y''=y$ is a
single point. Varying $x\in X$, we see that the lines $l_x$ sweep out
a dense subset in $Y$ while by our hypothesis $Y'$ is smooth and the
embedded tangent spaces $T_{Y',y}$ are contained in the subvariety
$SY'\subsetneq Y$; hence we may assume that $l\not\subset T_{Y',y}$.
The line $l$ is blown down by the map $\phi$ defined by the polar
linear system $\cQ$, hence a general quadric from $\cQ$ meets $l$
only at $y$.

\begin{lemma}\label{l.linear} Suppose that $\<Y''\>\subset Y'$. Then
$Y$ is polar defective.
\end{lemma}

\begin{proof} Put $P_x=\<l_x,Y''\>$. As $x$ varies in $X$, the lines
$l_x$ sweep out $Y$, whence, for a general $x\in X$,
$P_x\supsetneq\<Y''\>$. Thus $\<Y''\>$ is a hyperplane in $P_x$, and
so by our hypothesis $\<Y''\>$ is a fixed component of the
restriction $\cQ_{P_x}$ of the polar system of quadrics on $P_x$ and
$\phi_{|P_x}$ is a linear projection. Since $\phi(l_x)=x$, we
conclude that the restriction of the polar map $\phi$ on $P_x$
coincides with the projection from the point $y$ and
$\dim{\phi(P_x)}=\dim{\<Y''\>}=\dim{P_x}-1$.

Denote by $\mathbf P$ the closure of a union of all such $P_x$. Since
$P_x\supset l_x$ and the locus of $l_x$ is dense in $Y$, we have
$\mathbf P\supset Y$. On the other hand, $P_x\not\subset Y$ for a
general $x\in X$ since otherwise $Y$ would be a cone with vertex
$\<Y''\>$. Hence $\mathbf P=\P^N$, and so $\dim{\phi(\P^N)}=N-1$ and
$\phi(\P^N)$ is a hypersurface in $\P^{N\,{}^*}$ containing $X$.
\end{proof}

From now on we assume that $\<Y''\>\not\subset Y'$ and so
$T_{Y',y}\not\subset Y'$ for an arbitrary point $y\in Y''$.

\medskip

We skip the proof of the following elementary lemma on linear systems
of conics in a plane.

\begin{lemma}\label{l.P}
Let $P=\P^2$ be a plane, let $l,l'\subset P$ be distinct lines, and
let $y=l\cap l'$. Let $\cQ_P$ be a linear system of conics in $P$
whose general member meets $l$ and $l'$ only at $y$, and let $\phi_P$
be the rational map defined by $\cQ_P$. Then all members of $\cQ_P$
are unions of pairs of lines through $y$. Furthermore, if $\cQ_P$ is
nonconstant, then $\dim{\phi_P(P)}=1$ and $\phi_P=\psi_P \circ\pi_P$,
where the projection $\pi_P:P\dasharrow\P^1$ is the projection with
center $y$ and $\psi_P$ is a finite map of degree at most two defined
by a subsystem of the linear system $|\cO_{\P^1}(2)|$.
\end{lemma}

Let $y\in Y''$ be a general point, and let $X_y$ be the closure of
the set of general points $x\in X$ for which the line $l=l_x=
\pi(\cP_x)=\gamma^{-1}(x)$ meets $Y'$ in the (unique) point $y$. By
definition, the hyperplane ${}^\perp y$ is tangent to $X$ along $X_y$
(i.e. at the points of $X_y\cap\Sm{X}$) and $X_y$ is a fiber of the
dominant rational map $X\dasharrow Y''$ sending a general point $x\in
X$ to the (unique) point $y=l_x\cap Y''$, whence $\dim{X_y}=n-r''$,
where $r''=\dim{Y''}$. Let $P_y$ denote the closure of a union of
lines $l_x$, where $x\in X_y$ is a general point. Then $\gamma(P_y)=
X_y$ and $\dim{P_y}=n-r''+1=N-r''-1$. Let $r'$ be the dimension of
the irreducible component of $Y'$ containing $Y''$. Then
$r'=\dim{L_y}$ where $L_y=T_{Y',y}$, and it is clear that $r'\ge
r''$. Denote by $\cQ_y$ the restriction of the polar linear system
$\cQ$ to $P_y$, let $\phi_y=\phi_{|P_y}:P_y\dasharrow X_y$ be the map
defined by $\cQ_y$, and let $\text{\rm B}_y\subset P_y$ be the
fundamental subscheme (i.e. the scheme theoretic intersection of all
members) of $\cQ_y$.

\begin{lemma}\label{l.l} Suppose that $Y$ is not polar defective. Then
\begin{itemize}
\item[(i)] $P_y=\P^{N-r''-1}$ is a linear subspace contained in $Y${\rm;}
\item[(ii)] $L_y\subsetneq P_y$ and $r'+r''\le n${\rm;}
\item[(iii)] $X_y\subset X$ is a linear subspace, $\dim{X_y}=r'=r''
=\dim{L_y}=\frac n2$ and $\dim{P_y}=\frac n2+1${\rm;}
\item[(iv)] The fundamental subscheme $\text{\rm B}_y\subset P_y$ is a
cone with vertex $y$ properly contained in the hyperplane
$L_y${\rm;}
\item[(v)] $\phi_y=\psi_y\circ\pi_y$, where $\pi_y:P_y\dasharrow
\P^\frac n2$ is the projection with center $y$ and $\psi_y:
\P^\frac n2\dasharrow X_y=\P^\frac n2$ is a birational map
defined by a linear system of quadrics whose fundamental
subscheme is properly contained in the hyperplane
$\Lambda_y=\pi_y(L_y)$.
\end{itemize}
\end{lemma}

\begin{proof} (i) Let $l=l_x$, $l'=l_{x'}$, where $(x,x')$ is a pair
of general points in $X_y$, and put $P=\<l,l'\>$. Then $P,l,l'$ and
$\cQ_P=\cQ_{|P}$ satisfy the hypotheses of Lemma~\ref{l.P}. Let
$\mathbf P$ denote the closure of the locus of planes $P$. Since
$P\supset l=l_x$ and the locus of $l_x$ is dense in $Y$, $\mathbf P$
contains $Y$. If a general plane $P$ from our family is not contained
in $Y$, then $\mathbf P=\P^N$, and from Lemma~\ref{l.P} it follows
that $Y$ is polar defective, contrary to our hypothesis. Thus
$P\subset Y$ and therefore $P\subset P_y$, from which it follows that
the secant variety $SP_y$ is contained in $P_y$, whence $P_y\subset
Y$ is a linear subspace.

(ii) Since $x\in X_y$ is general, $l=l_x\not\subset L_y$. Let $l'\ni
y$ be a general tangent line to $Y'$ at the point $y$. Since $Y'$ is
smooth, from Proposition~\ref{p.Sec} it follows that $l'\subset SY'
\subset Y$. We may assume that $l'\not\subset Y'$ since otherwise
$L_y\subset Y'$, the irreducible component of $Y'$ containing $Y''$
is linear, and $Y$ is polar defective by Lemma~\ref{l.linear}.
Clearly, $l'$ is blown down by $\phi_y$ and a general quadric from
$\cQ_y$ meets $l'$ only at $y$. Let $P=\<l,l'\>$ be the plane spanned
by the lines $l$ and $l'$. Then $P,l,l'$ and $\cQ_P=\cQ_{|P}$ satisfy
the hypotheses of Lemma~\ref{l.P} and the same argument as in~(i)
shows that $Y$ is polar defective unless $P\subset Y$. Thus $P\subset
Y$, and from Lemma~\ref{l.P} it follows that $P\subset P_y$. From
this it follows that $L_y \subsetneq P_y$, and since $r'=\dim{L_y}\le
\dim{P_y}-1=n-r''$, we conclude that $r'+r''\le n$.

(iii) Take any point $v\in P_y\cap\Sm{Y}$. Then $P_y\subset T_{Y,v}$,
and so $\gamma(v)={}^\perp T_{Y,v}\in{}^\perp P_y$, where by (i)
${}^\perp P_y\subset\P^{N\,{}^*}$ is a linear subspace of dimension
$r''$. Varying $v$ in $P_y$, we see that $X_y=\gamma(P_y)\subset
{}^\perp P_y$ and $n-r''=\dim{X_y}\le r''$, i.e. $r''\ge\frac n2$.
Since by (ii) $r'+r''\le n$, we conclude that $X_y={}^\perp P_y
\subset X$ is a linear subspace and $\dim{X_y}=r''=r'=\frac n2$,
$\dim{P_y}=\frac n2 +1$.

(iv) From Lemma~\ref{l.P} it follows that the members of the linear
system $\cQ_y$ are cones with vertex $y$, hence $\Supp{\text{\rm
B}_y}$ is a cone with vertex $y$ in $L_y$. The linear system $\cQ_y$
does not have fixed components since, being a hyperplane in $P_y\cap
Y'$ passing through $y$, such a component should necessarily coincide
with $L_y$, which, by Lemma~\ref{l.linear}, contradicts the
hypothesis of the lemma.

To show that $\text{\rm B}_y$ is contained in $L_y$ it suffices to
check that, for any point $v\in\text{\rm B}_y$, $v\ne y$ the
(embedded projective) tangent space $T_{\text{\rm B}_y,v}$ is
contained in $L_y$. But if it were not so, then there would exist a
line $m\subset T_{\text{\rm B}_y,v}$ such that $m\cap L_y=v$,
$y\notin m$, and $m$ is contracted by $\phi_y$. Thus to prove~(iv) it
remains to show that if $m\subset P_y$ is a line such that $y\notin
m$ and $m\not\subset L_y$, then $\phi_y(m)$ is a curve. Let $P$ be
the plane spanned by $m$ and $y$. The linear system $\cQ_P=\cQ_{|P}$
is nonconstant since otherwise, by the above, the intersection of $P$
with $\text{\rm B}_y$ would consist of two (possibly coinciding)
lines passing through $y$ and the plane $P \supset m$ would be
contained in $L_y$, contrary to our assumption. Our claim now follows
from Lemma~\ref{l.P} according to which $\phi_y(m)= \phi_y(P)$ is a
curve.

(v) It is well known that, for a general point $x\in X$, the fiber
$\gamma^{-1}(x)$ is a line (cf. \cite[ch.~I, Theorem~2.3\,c)]{Z1}),
whence $\psi_y$ is birational. The other claims follow from~(i),
(iii), (iv), and Lemma~\ref{l.P}.

\end{proof}

To complete the proof of Theorem~\ref{t.PolarDefective} it remains to
show that the situation described in Lemma~\ref{l.l} does not
actually occur. We will do that by proving a general result on linear
systems of hypersurfaces in projective space defining a birational
map onto a linear space of the same dimension (or, equivalently, a
birational automorphism of the ambient space; such automorphisms are
called {\it Cremona transformations}\/). To formulate this result it
is convenient to use the following definition.

Let $V$ be a projective variety, and let $\text{\rm B}\subset V$ be a
subscheme. A linear system of divisors $\cD$ on $V$ is called {\it
complete modulo $\text{\rm B}$\/} if $\cD=|D-B|$, where $|D-B|$ is
the linear system of {\it all\/} effective divisors on $V$ that
contain $B$ and are linearly equivalent to divisors from $\cD$.

Various versions of the following lemma are known to experts, but we
include a proof for the reader's convenience.

\begin{lemma}\label{l.qm}\hfill
\begin{itemize}
\item[(i)] A linear system of hypersurfaces of degree $d$ in $\P^N$
defining a birational map $\P^N\dasharrow\P^N$ is complete modulo
its fundamental subscheme $\text{\rm B}${\rm;}
\item[(ii)] Suppose that the fundamental subscheme $\text{\rm B}$ of a
linear system of hypersurfaces of degree $d$ in $\P^N$ defining a
birational map $\phi:\P^N\dasharrow\P^N$ is contained in a
hypersurface of degree less than $d$. Then the linear system has
a fixed component of degree $d-1$, the residual system is the
complete linear system of hyperplanes in $\P^N$, and $\phi$ is a
linear automorphism.
\end{itemize}
\end{lemma}

\begin{proof} (i) We reformulate the lemma in terms of the Veronese
variety $V=v_d(\P^N)\subset\P^{\nu}$, $\nu=\binom{N+d}d-1$. Let
$L\subset\P^{\nu}$ be the linear subspace of codimension $N+1$ in
$\P^{\nu}$ corresponding to our linear system, then $v_d(\text{\rm
B)}=L\cdot V$. If our linear system were not complete modulo
$\text{\rm B}$, there would exist a linear subspace $L'\subsetneq L$
such that $v_d(\text{\rm B})=L'\cdot V$. Let $\Lambda\subset L$ be a
linear subspace of dimension $\codim_L{L'}-1$ disjoint from $L'$,
hence also from $V$, and let $\pi_L$ and $\pi_{L'}$ be the
projections with centers $L$ and $L'$ respectively. Then
$\pi_L=\pi_{\Lambda'}\circ\pi_{L'}$, where $\Lambda'=
\pi_{L'}(\Lambda)$. The variety $V'=\pi_{L'}(V)$ is nondegenerate in
the linear space $\pi_{L'}(\P^{\nu})$, and by our hypothesis the map
$\pi_{\Lambda'}:V'\dasharrow\P^N$ is birational. On the other hand,
$\pi_{L'}^{-1}(V')=S(L',V)$ and $L\cap S(L',V)=S(L',L\cap V)=S(L',L'
\cap V)=L'$, hence $\Lambda\cap\pi_{L'}^{-1}(V')=\emptyset$ and, by
projection formula, $\Lambda'$ does not meet $V'$. Thus
$\pi_{\Lambda'}:V'\to\P^N$ is a finite morphism of degree
$\deg{V'}>1$, a contradiction.

(ii) If $\text{\rm B}\subset H_{d'}$, where $H_{d'}\subset\P^N$ is a
hypersurface of degree $d'<d$, then all the hypersurfaces of the form
$H_{d'}\cup W_{d-d'}$, where $W_{d-d'}$ is an arbitrary hypersurface
of degree $d-d'$, also contain $\text{\rm B}$. The dimension of the
linear system of such hypersurfaces (equal to $\binom{N+d-d'}{d-d'}
-1$) is larger than $N$ for $d'<d-1$ and equal to $N$ for $d'=d-1$.
The first case is ruled out by~(i), and in the second case the system
can only be complete if $\text{\rm B}=H_{d-1}$.
\end{proof}

Combining Lemma~\ref{l.l}\,(v) with Lemma~\ref{l.qm}\,(ii), we see
that $Y$ is polar defective. This completes the proof of
Theorem~\ref{t.PolarDefective}.
\end{proof}

\begin{remark}\label{r.sing} The assumption that $Y'$ is smooth in
the statement of Theorem~\ref{t.PolarDefective} is not necessary. The
third named author proved that the theorem holds even without this
assumption. However, our Main Theorem~\ref{t.Severi} is false without
the smoothness hypothesis (condition~b) in its statement) as is shown
by Example~\ref{e.SectionSecSeveri}\,(ii). Thus this assumption is
not restrictive for the purposes of the present paper and we do not
give a proof of Theorem~\ref{t.PolarDefective} in full generality.
\end{remark}

In Section~\ref{s.3} we will use Theorem~\ref{t.PolarDefective} via
its

\begin{corollary} \label{c.main}
Let $Y \subset \P^N$ be an irreducible cubic hypersurface.  Assume
that
\begin{itemize}
\item[a)] $Y$ is dual defective, but not polar defective{\rm;}
\item[b)] the singular locus $Y'$ of $Y$ is smooth.
\end{itemize}
Then there exists an irreducible component $Y'_0\subset Y'$ such that
$Y=SY'_0$.
\end{corollary}

\begin{proof}
From Theorem~\ref{t.PolarDefective} it follows that $Y=SY'$. Suppose
that the secant variety of an arbitrary component of $Y'$ is a proper
subvariety of $Y$. Then there exist distinct irreducible components
$Y_1,Y_2\subset Y'$ such that $Y=S(Y_1,Y_2)$, where $S(Y_1,Y_2)$ is
the {\it join\/} of $Y_1$ and $Y_2$, i.e. \!the closure of the
subvariety in $\P^N$ swept out by the lines $\<y_1,y_2\>$, where
$y_1$ (resp. $y_2$) runs through the subset of general points in
$Y_1$ (resp. \!$Y_2$). Let $X^n=Y^*$, and let $x\in X$ be a general
point. The lines of the form $\<y_1,y_2\>$ are contracted by the
Gauss map $\gamma:Y\dasharrow X$, and so $N\ge n+2$. It was already
shown that $\cP'_x\subset\cP_x$ is a hypersurface whose degree does
not exceed $2$, and since under our assumptions $\cP'_x$ is not
linear, $\cP'_x=\cP_x\cap(Y_1\cup Y_2)$ is a union of two hyperplanes
in the $(N-n-1)$-dimensional linear space $\cP_x$. Therefore $Y_1\cap
Y_2 \cap\cP_x\ne\emptyset$ provided that $N-n-1>1$. Since $Y'$ is
smooth, $Y_1\cap Y_2=\emptyset$, and so $N=n+2$ and
$\cP_x=\gamma^{-1}(x)$ is a line. Hence a general point $y\in Y$  is
contained in a unique secant line $\<y_1,y_2\>$, $y_1\in Y_1$,
$y_2\in Y_2$.

Let $z_0:\cdots:z_{2N+1}$ be homogeneous coordinates in
$\P^{2N+1}$. Consider two copies $\Lambda_1$ and $\Lambda_2$ of
$\P^N$ embedded as disjoint linear subspaces in $\P^{2N+1}$ as
follows:
$$
\Lambda_1=\{z_{N+1}=\cdots=z_{2N+1}=0\},\qquad
\Lambda_2=\{z_0=\cdots=z_N=0\}.
$$
Let $\tilde Y_1\subset\Lambda_1$ and $\tilde Y_2\subset\Lambda_2$
be the corresponding embeddings of $Y_1$ and $Y_2$, and let
$\tilde Y= S(\tilde Y_1,\tilde Y_2)$ be their join. It is easy to
see that $\deg{\tilde Y}=\deg{\tilde Y_1}\cdot\deg{\tilde Y_2}$.

Put $\Lambda=\{z_{N+1}=z_0,z_{N+2}=z_1,\cdots,z_{2N+1}=z_N\}$, and
let $\pi_{\Lambda}:\P^{2N+1}\dasharrow\P^N$ denote the projection
with center $\Lambda$. Since $Y_1\cap Y_2=\emptyset$, $\Lambda\cap
\tilde Y=\emptyset$ and $\pi_{\Lambda}:\tilde Y\to Y$ is a finite
morphism. The degree of $\pi_{\Lambda}$ is equal to the number of
pairs $(y_1,y_2)\in Y_1\times Y_2$ such that $\<y_1,y_2\>\ni y$, and
so from the above it follows that $\deg{\pi_{\Lambda}}=1$. Thus
$\pi_{\Lambda}$ is birational, $3=\deg{Y}=\deg{\tilde
Y}=\deg{Y_1}\cdot\deg{Y_2}$, and either $Y_1$ or $Y_2$ is linear,
contrary to the hypothesis that $Y$ is not a cone.
\end{proof}

\section{Cubic hypersurfaces admitting nonzero prolongations} \label{s.3}

To study group actions on projective algebraic varieties it is often
more convenient to consider the corresponding affine cones. We start
with interpreting in the affine language some of the notions
introduced in the preceding section.

Let $\P^N=\P W$, where $W$ is an $(N+1)$-dimensional complex vector
space, and let $F\in\Sym^3{W^*}$ be an irreducible cubic form. For
$w\in W$, we denote by $F_{ww}\in W^*$ (resp. $F_w\in\Sym^2{W^*}$)
the linear function (resp. quadratic form) on $W$ defined by
$u\mapsto F(w,w,u)$ (resp. $(u, v)\mapsto F(w,u,v)$). The affine
polar map $\hat{\phi}:W\dasharrow W^*$ is given by $w\mapsto F_{ww}$
and the Hessian by $F_w$. Let $\Hat Y=\{w\in W\mid
F(w,w,w)=0\}\subset W$ be the affine cubic hypersurface defined by
$F$. The indeterminancy locus of $\hat\phi$ coincides with the
singular locus $\Sing{\Hat Y}=\{w\in W\mid F_{ww}=0\}$. For any
$y\in\Sm{\Hat Y}$ we have
\begin{equation}\label{e.t}
T_{\Hat Y,y}=\Ker{F_{yy}}=\{v\in W\mid F(y,y,v)=0\}.
\end{equation}
The affine Gauss map $\Hat\gamma:\Hat Y\dasharrow W^*$, $\Hat\gamma=
\Hat\phi_{|\Hat Y}$ maps a smooth point $y\in\Hat Y$ to the
hyperplane $T_{\Hat Y,y}$. The closure of $\Hat\gamma(\Hat Y)$  in
$W^*$ is called the affine dual variety of $\Hat Y$ and is denoted by
$\Hat X$.

Denote by $\aut{\Hat Y}\subset\gl{W}$ the Lie algebra of
infinitesimal linear automorphisms of $\Hat Y$:
$$
\aut{\Hat Y}=\{g\in\gl{W}\mid g(y)\in T_{\Hat Y,y}\quad\forall\,y\in
\Sm{\Hat Y}\}.
$$
By \eqref{e.t}, this can be rewritten as follows:
$$
\aut{\Hat Y}=\{g\in\gl{W}\mid F(g(y),y,y)=0\quad\forall\,y\in
\Sm{\Hat Y}\}.
$$

A prolongation of $\aut{\Hat Y}$ is an element $A\in
\Hom{(\Sym^2W,W)}$ such that $A(w,\cdot)\in\aut{\Hat Y}$ for all
$w\in W$. The vector space of all prolongations of $\aut{\Hat Y}$
will be denoted by $\autp{\Hat Y}$. In other words, an element $A\in
\Hom{(\Sym^2W,W)}$  is contained in $\autp{\Hat Y}$ if and only if
\begin{equation}\label{e.p}
 F(A(y,w),y,y)=0\quad\forall\,y\in\Sm{\Hat Y},\ \forall\,w\in W.
\end{equation}

\begin{proposition}\label{p.DP}
Let $Y\subset\P W$ be an irreducible cubic hypersurface. If
$\autp{\Hat Y}\neq0$, then $Y$ is dual defective.
\end{proposition}

\begin{proof}
The differential of $\hat\phi$ at a point $w\in W$ is given by the
Hessian of $F$ or, equivalently, by the form $F_w$. Since, for $y\in
\Sm{\Hat Y}$, the affine Gauss fiber ${\hat\gamma}^{-1}
(\hat\gamma(y))$ coincides with the linear subspace $\<y,
\Ker{d_y\hat\phi}\>$, to prove the proposition it suffices to show
that, for a general point $y\in\Hat Y$, the form $F_y$ is degenerate,
i.e. there exists a nonzero vector $v\in W$ such that $F(y,v,w)=0$
for a general point $w\in W$.

Let $A\in\autp{\Hat Y}$ be a nonzero prolongation.
From~\eqref{e.p} it follows that, for any $u\in W$, the cubic forms
$F(A(u,w),w,w)$ and $F(w,w,w)$ define the same hypersurface
$Y\subset\P^N$. Therefore, for all $u,w\in W$,
\begin{equation}\label{e.f}
F(A(u,w),w,w)=\lambda(u)F(w,w,w),
\end{equation}
where $\lambda=\lambda^A\in W^*$. Replacing $w$ by $tu+sw$, we get an
identity of the form $\sum\limits_{i=0}^3c_i(u,w)t^is^{3-i}\equiv0$,
hence $c_i=0$ for all $i$, $0\le i\le3$. In particular, $c_0=0$ is
just~\eqref{e.f}, $c_3=0$ is a consequence of~\eqref{e.f} and
\begin{align}
c_1=0\,&\Rightarrow\, F(A(u,u),w,w)+2F(A(u,w),u,w)=3\lambda(u)F(u,w,w),
\label{e.1}\\
c_2=0\,&\Rightarrow\,2F(A(u,u),u,w)+F(A(u,w),u,u)\,\,=3\lambda(u)F(u,u,w).
\label{e.2}
\end{align}

Assuming now that $u=y\in\Sm{\Hat Y}$ and using~\eqref{e.p}, we see
that from~\eqref{e.2} it follows that $2F(A(y,y),y,w)=
3\lambda(y)F(y,y,w)$ for all $y\in\Sm{\Hat Y}$ and $w\in W$. In other
words, $F(y,v,w)=0$ for all $w\in W$, where $v=A(y,y)-\mu(y)y\in
\Ker{d_y\hat\phi}$ and $\mu=\frac32\lambda$. Thus to prove the
proposition it remains to show that $v\ne0$.

Suppose to the contrary that $A(y,y)=\mu(y)y$ for all $y\in\Sm{\Hat
Y}$ with $\mu=\mu^A\in W^*$. Since $\overline{\Sm{\Hat Y}}=\Hat Y$,
the irreducible cubic $\Hat Y$ is contained in the quadric defined by
the equation $A(w,w)-\mu(w)w=0$, hence $A(w,w)-\mu(w)w=0$ for all
$w\in W$, from which it follows that
\begin{multline}
A(u,w)=\frac{A(u+w,u+w)-A(u,u)-A(w,w)}2\\
=\frac{\mu(u)w+\mu(w)u}2,\quad\forall\,u,w\in W.\label{e.g}
\end{multline}
Since $A\ne0$, from~\eqref{e.g} it follows that $\mu$ and $\lambda$
are nonzero elements of $W^*$. Substituting \eqref{e.g}
in~\eqref{e.1}, we get
\begin{multline}
F(A(u,u),w,w)+2F(A(u,w),u,w)\\
=\mu(u)F(u,w,w)+\mu(u)F(u,w,w)+\mu(w)F(u,u,w)\\
=2\mu(u)F(u,w,w)+\mu(w)F(u,u,w)\\
=3\lambda(u)F(u,w,w)=2\mu(u)F(u,w,w),
\end{multline}
hence $F_{uu}\equiv0$ and we arrive at a contradiction since $F_{uu}$
vanishes only on $\Sing{\Hat Y}$.
\end{proof}

\begin{remark}\label{r.H} Proposition~\ref{p.DP} was proved as Corollary~4.5
in \cite{H} under the additional assumption that $Y$ is not polar
defective. However from Proposition~\ref{p.PD} it follows that polar
defective hypersurfaces are a fortiori dual defective.
\end{remark}

Existence of nonzero prolongations of $\aut{\hat V}$ imposes strong
geometric restrictions on a smooth nondegenerate projective variety
$V\subset\P W$. For example, from \cite{HM} it follows that there
exist lots of $\C^*$-actions on $V$ the closures of whose orbits are
conics in $\P W$. In particular, this implies that $V$ is conic
connected. Complete classification of such varieties $V$ is carried
out in \cite{FH1} and \cite{FH2}. In the case when $V$ is linearly
normal, this classification is as follows.

\begin{theorem}[Theorem 7.13 in \cite{FH2}]\label{t.FH}
Let $V\subsetneq\P W$ be a nondegenerate irreducible smooth linearly
normal variety such that $\autp{\hat V}\ne0$. Then $V\subset\P W$ is
projectively equivalent to one of the following varieties.
\begin{itemize}
\item[(1)] A rational homogeneous variety from the following
list{\,\rm:}
$$
v_2(\P^m),\quad\P^a\times\P^b,\quad G(1,m),\quad Q^m,\quad
\S_5,\quad\mathbb{OP}^2,
$$
viz. the second Veronese embedding of $\P^m$, the Segre embedding
of $\P^a\times\P^b$, the Pl\"ucker embedding of the Grassmann
variety of lines in $\,\P^m$, the $m$-dimensional nonsingular
quadric, the $10$-dimensional spinor variety $\S_5\subset\P^{15}$
parametrizing $4$-dimensional linear subspaces from one family on
the $8$-dimensional quadric, and the $16$-dimensional Cayley
plane in $\P^{26}${\rm;}
\item[(2)] A smooth section of $G(1,4)\subset\P^9$ by a linear
subspace of codimension $1$ or $\,2$ in $\,\P^9${\rm;}
\item[(3)] The section of $\,\S_5\subset\P^{15}$ by a linear
subspace $L$ of codimension $1$, $2$ or $3$ in the ambient
$\P^{15}$ which is general in the sense that $L\cap\S_5$ is
smooth and if $\codim_{\P^{15}}{L}>1$,  $L$  contains
one of the $10$-dimensional family of $\P^4$'s lying on
$\S_5${\rm;}
\item[(4)] The blowup ${\rm Bl}_{\P^s}(\P^m)$ embedded by the linear
system of quadric hypersurfaces in $\P^m$ passing through $\P^s$.
\end{itemize}
\end{theorem}

Now a straightforward computation of the dimension of the secant
varieties of the varieties from the above list yields the following

\begin{corollary} [Theorem 2.1 in \cite{H}]\label{t.SeveriProlong}
Let $V\subset\P W$ be a nondegenerate irreducible linearly normal
smooth subvariety with $\autp{\hat V}\ne0$. If $SV\subset\P W$ is a
hypersurface, then $V\subset\P W$ is a Severi variety.
\end{corollary}

\begin{example}\label{e.nonzeroprolongation}
Let us explain why the varieties listed in Theorem~\ref{t.FH}\,(1)
have nonzero prolongations. Let $V^n\subset\P^N$ be a variety from
this list. By Theorem~3.8 in Chapter~III of \cite{Z1}, $V^n=
\psi(\P^n)$, where the rational map $\psi:\P^n\dasharrow\P^N$ is
given by the linear system of quadrics passing through a base locus
$B\subset\P^{n-1}\subset\P^n$, where $B=\emptyset\,$ for
$V=v_2(\P^m)$, $B=\P^{a-1}\sqcup\P^{b-1}$ for $V=\P^a\times\P^b$
(meaning that $B=\P^{a-1}\cup\P^{b-1}$ and $\P^{a-1}\cap\P^{b-1}=
\emptyset$), $B=\P^1\times\P^{m-2}$ (Segre embedding) for $V=G(1,m)$,
$B=Q^{m-2}$ for $V=Q^m$, $B=G(1,4)$ for $V=\S_5$, and $B=\S_5$ for
$V=\mathbb{OP}^2$. In coordinates, let $\P^{n-1}=\P U$, $\P^n=
\P(\C\oplus U)$ and $\P^N=\P W$, where $W=\C\oplus U\oplus K$, where
$U$ and $K$ are complex vector spaces. The linear system of quadrics
defining $\psi$ contains the reducible quadrics formed by $\P^{n-1}$
and an arbitrary hyperplane in $\P^n$, and $\psi(t:u)=
(t^2:tu:\sigma(u,u))$, where $\sigma:\Sym^2{U}\to K$ corresponds to
the linear system of quadrics in $\P^{n-1}$ passing through $B$. The
map $\psi$ identifies the affine space $U$ with an open subset of $V$
and, in the coordinates $(w_0,w_1,w_2)$ corresponding to the
decomposition $W=\C\oplus U\oplus K$, $V$ is locally defined by the
equations $w_0w_2=\sigma(w_1,w_1)$.

Now define $A:\Sym^2W\to W$ by
$$
A\bigl((w_0,w_1,w_2),(w'_0,w'_1,w'_2)\bigr)=\biggl(w_0 w'_0,
\frac{w_0w'_1+w'_0w_1}{2},\sigma(w_1,w'_1)\biggr).
$$
Then $A\ne0$ and in \cite[Proposition~7.11]{FH2} it is shown that,
for any $v\in\Hat V$ and $w\in W$, one has $A(v,w)\in T_{\Hat V,v}$,
i.e. $A\in \autp{\hat V}$.
\end{example}

\begin{remark}\label{r.FH}
In \cite[Proposition 3.4]{FH1} it is shown that for the varieties $V$
from Theorem~\ref{t.FH}\,(1) one actually has $\autp{\hat V}\simeq
W^*$.
\end{remark}

\begin{example}\label{e.TangentialSeveri}
Now let $V=V_i^{n_i}\subset\P^{N_i}$, $n_i=\dim{V_i}=2^i$, $N_i=
\frac{3n_i}2+2$ be the $i$-th Severi variety, $1\le i\le4$ (cf.
Remark~\ref{r.PD}), and let $Y=Y_i=SV_i$ be the corresponding cubic
hypersurface. Being special cases of Theorem~\ref{t.FH}\,(1), Severi
varieties have nonzero prolongations constructed in
Example~\ref{e.nonzeroprolongation}. Furthermore, the description of
the base loci $B$ given in Example~\ref{e.nonzeroprolongation} shows
that $\dim{U_i}=\dim{V_i}= n_i$, $Q=Q_i=\psi_i(\P U_i)\subset V_i$ is
a smooth quadric of dimension $n_i/2=2^{i-1}$, $y=y_i=\psi_i(\C)$ is
a point in $\<Q\,\>\setminus Q$, where $Q$ is the entry locus of $y$,
i.e. the locus of points $v\in V$ for which there exists a point
$v'\in V$ such that $\<v,v'\>\ni y$, and the hyperplane $H=T_{Y,y}$
is tangent to $V$ along $Q$ and to $Y$ along $\<Q\,\>$ (cf.~\cite[ch.
IV, Theorem~2.4]{Z1}). Conversely, by \cite[ch. III, Theorem~3.8]{Z1}
(or by the formula for $\psi$ given in
Example~\ref{e.nonzeroprolongation}), if $\pi:\P W \dasharrow\P^n$ is
the projection with center $\<Q\,\>=\P^{\frac n2+1}$, then
$\pi(H)=\P^{n-1}\subset\P^n$, $\pi\circ\psi=\id$,
$\pi_{|V}:V\dasharrow\P^n$ is the birational isomorphism inverse to
$\psi$, and $\pi_{|V\setminus H\cap V}:V\setminus H\cap V
\overset\sim\to U=\P^n\setminus \P^{n-1}$. Furthermore, if
$x={}^{\perp}H=\gamma(y)\in X=Y^*$, then, by duality,
${}^{\perp}T_{X,x}=\<Q\,\>$ and ${}^{\perp}\<Q\,\>=T_{X,x}\subset
SX=V^*$. Recall that $X$ is also a Severi variety (projectively
isomorphic to $V$) and the linear subspaces $T_{X,x}$ sweep out $SX$.
Hence, in particular, {\it a hyperplane $L$ is tangent to $V\!$ if
and only if it contains one of the quadratic entry loci of type $Q$}.

Fixing $Q=\psi(\P U)$ as above, we observe that the projection~$\pi$
establishes a projective equivalence between the hyperplanes in $\P
W$ passing through $Q$ and the hyperplanes in $\P(\C\oplus U)$, and
this equivalence is respected by the map $\psi:\P(\C\oplus U)
\dasharrow V$. Thus the argument in
Example~\ref{e.nonzeroprolongation} shows that for any $L\supset Q$
one has $\autp{\Hat V_L}\ne0$, where $V_L=L\cap V$.
\end{example}

The following result is well-known, but we include a proof for the
sake of completeness.
\begin{lemma} \label{l.linearnormal}
Let $V\subset \mathbb{P}^N$ be a regular linearly normal variety, let
$L\subset\P^N$ be a hyperplane, and let $V_L=L\cap V$. Then $V_L$ is
linearly normal in $L$.
\end{lemma}
\begin{proof}
Let $\mathcal I_V$  (resp. $\mathcal I_{V_L}$) be the ideal sheaf
defining $V$ in $\P^N$ (resp. $V_L$ in $L$). The exact sequence
$$
0\to H^1(\P^N,\mathcal I_V)\to H^1(\P^N,\mathcal O_{\P^N})\to
H^1(V,\mathcal O_V)\to H^2(\P^N,\mathcal I_V)\to
H^2(\P^N,\mathcal O_{\P^N})=0
$$
yields $H^1(\P^N,\mathcal I_V)=H^2(\P^N,\mathcal I_V)=
H^1(V,\mathcal O_V)=0$, hence from the exact sequence
$$
H^1(\P^N,\mathcal I_V)\to H^1(\P^N,\mathcal I_V(1))\to
H^1(L,\mathcal I_{V_L}(1))\to H^2(\P^N,\mathcal I_V)
$$
it follows that $H^1(L,\mathcal I_{V_L}(1))\simeq
H^1(\P^N,\mathcal I_V(1))=0$.
\end{proof}

\begin{corollary}\label{c.phs}
Let $V\subset\P^N$ be a Severi variety, let $L\subset\P^N$ be a
hyperplane, and put $V_L=L\cap V$. The following conditions are
equivalent{\rm:}
\begin{itemize}
\item[a)] $V_L$ is singular {\rm(}i.e. $L$ is tangent to $V$,
viz. ${}^{\perp}L\in V^*${\rm);}
\item[b)] $\autp{\Hat V_L}\ne0$.
\end{itemize}
\end{corollary}

\begin{proof} As we saw in Example~\ref{e.TangentialSeveri},
$\text{a)}\Rightarrow\text{b)}$. On the other hand, suppose that
$V_L$ is nonsingular, but $\autp{\Hat V_L}\ne0$. By
Lemma~\ref{l.linearnormal}, $V_L$ is linearly normal, hence $V_L$
satisfies the hypotheses of Theorem~\ref{t.FH}. However, the list of
varieties in Theorem~\ref{t.FH} does not contain hyperplane sections
of Severi varieties.
\end{proof}

\begin{proposition}\label{p.equalAut}
Let $V\subset\P^N=\P W$ be an irreducible smooth variety, and let $Y=
SV$ be its secant variety. Suppose that $V$ is an irreducible
component of $Y'=\Sing{Y}$. Then $\aut{\Hat Y}=\aut{\Hat V}$ and
$\autp{\Hat Y}=\autp{\Hat V}$.
\end{proposition}

\begin{proof}
Let $g\in\aut{\Hat V}\subset\gl{W}$, and let $\xi_g\subset\GL{\,W}$
be the $1$-parameter subgroup generated by $g$. Since ${\Hat V}$ is
invariant with respect to $\xi_v$, the same is true for $\Hat Y=
S\Hat V$, hence $v\in\aut{\Hat Y}$. Conversely, continuous families
of automorphisms of $\Hat Y$ preserve every irreducible component of
its singular locus $Y'$, which proves the first claim. The second
claim follows from the first one by the definition of prolongations.
\end{proof}

\begin{corollary}\label{c.pSeveri}
Let $V\subset\P^N$ be a Severi variety, and let $Y=SV$ be its secant
variety. Then $\autp{\Hat Y}\ne0$.
\end{corollary}
\begin{proof} It is well known that for Severi varieties $Y'=V$.
Hence the claim follows from Example~\ref{e.nonzeroprolongation} and
Proposition~\ref{p.equalAut}.
\end{proof}

We are now ready to prove the Main Theorem.

\begin{theorem}\label{t.Severi}
Let $Y\subset\P W$ be an irreducible cubic hypersurface.  Assume that
\begin{itemize}
\item[a)] $Y$ is not polar defective{\rm;}
\item[b)] the singular locus $Y'$ is smooth{\rm;}
\item[c)] $\autp{\hat Y}\ne0$.
\end{itemize}
Then $Y$ is the secant variety of a Severi variety.
\end{theorem}

\begin{proof}
By Proposition~\ref{p.DP}, $Y$ is dual defective.  By
Corollary~\ref{c.main}, $Y=SY'_0\,$ for some irreducible component
$Y'_0\subset Y'$. By Proposition~\ref{p.equalAut}, $\autp{\hat
Y'_0}=\autp{\hat Y}\ne0$.

We claim that $Y'_0$ is linearly normal. In fact, if $Y'_0$ fails to
be linearly normal, then there exist a nondegenerate variety
$\tY'_0\subset\P^{N+1}$ and a point $z\in\P^{N+1}\setminus S\tY'_0$
such that $\pi(\tY'_0)= Y'_0$, where $\pi=\pi_z:\P^{N+1}\dasharrow
\P^N$ denotes the projection with center $z$.  Since $Y=SY'_0$ is a
hypersurface, $\tilde Y=S\tilde Y'_0$ has codimension $2$ in
$\P^{N+1}$. If $S^2 \tilde Y'_0\subsetneq\P^{N+1}$, then $S^2\tilde
Y'_0= S(\tilde Y'_0,S\tilde Y'_0)$ is a hypersurface, hence, for a
general point $\tilde y\in\tilde Y'_0$, $S(\tilde Y'_0,S\tilde Y'_0)=
S(\tilde y,S\tilde Y'_0)$ is a cone with vertex $\tilde y$, which is
impossible since $\tilde Y'_0$ is nondegenerate. Thus $S^2\tilde Y'_0
=\P^{N+1}$. It follows that the point $z$ is contained in a
(nonempty) family of (possibly osculating) planes $\<\ty'_1,\ty'_2,
\ty'_3\>$, where $\ty'_i\in\tY'_0$, $1\le i\le3$ are (possibly
coinciding) points. Put $y'_i=\pi(\ty'_i)$, $1\le i\le3$. Then
$l=\<y'_1,y'_2,y'_3\>=\pi(\<\ty'_1,\ty'_2,\ty'_3\>)$ is a trisecant
line of $Y'_0$, and so $l\subset Y'_0$ since $Y'$ is defined by
quadratic equations. But by our hypothesis $\pi:\tY'_0\to Y'_0$ is an
isomorphism, and so $\tilde l=\pi^{-1}(l)\subset\tY'_0$ is a line,
contrary to the assumption that the points $\ty'_1,\ty'_2,
\ty'_3\in\tilde l$ are not collinear. Thus $Y'_0$ is linearly normal
and Theorem~\ref{t.Severi} follows from
Corollary~\ref{t.SeveriProlong}.
\end{proof}

We conclude with giving examples showing that none of the
hypotheses~a)--c) in the statement of Theorem~\ref{t.Severi} can be
lifted.

\begin{example}\label{e.SectionSecSeveri}
Let $V=V_i\subset\P^N$, $N=N_i$,  $i=1,\dots,4$  be a Severi variety
of dimension $n=n_i$, and let $Y=SV$ be its secant variety. Then
$X=Y^*\simeq V$ and $V^*\simeq SX$ (projective equivalence, cf.
\cite[ch. IV, (2.5.4)]{Z1}). There are three types of hyperplanes
$L\subset\P^N$ with respect to $V$ and $Y$ corresponding to the
filtration $X\subset SX\subset\P^{N\,{}^*}=\P W^*$ or, equivalently,
to the three orbits $O_1,O_2,O_3$ of the action of $\GL_3$ on $\P
W^*$ characterized by the rank of nonzero Hermitian $3\times3$
matrices over composition algebras over $\C$ making up $W^*$ (so that
$X=O_1$, $SX\setminus X=O_2$ and $\P W^*\setminus SX= O_3$). Let
$L\in\P^{N\,{}^*}$ be a hyperplane, and put $V_L=L\cap V$, $Y_L=L\cap
Y$. It is easy to see that in each of the three cases the cubic $Y_L$
is irreducible. Note that, for $n>2$, $Y_L=SV_L$ because the entry
loci have dimension $\frac n2>0$ (the same is true for $n=2$ if by
$V_L$ we understand the scheme theoretic intersection $V\cdot L$ and
define $SV_L=TV_L$ accordingly). In the following examples we study
properties of the cubic hypersurfaces $Y_L$ depending on the position
of $L$.

\begin{itemize}
\item[(i)] Suppose that $L$ is a general hyperplane, i.e. $L\in O_3$.
Then $Y_L$ is irreducible, $\Sing{Y_L}=Y'_L=L\cap Y'=L\cap
V=V_L$, and from Corollary~\ref{c.phs} and
Proposition~\ref{p.equalAut} it follows that $\autp{\Hat Y_L}=0$.
Furthermore, $\pDef{Y_L}=0$ by
Proposition~\ref{p.HyperplaneSection}, and it is clear that
$Y'_L$ is smooth; so conditions~a) and b) of the Main Theorem are
satisfied. This example shows that condition~c) of the Main
Theorem cannot be lifted. It is worthwhile to note that in this
example $Y'_L=V_L$ is a homogeneous variety, and so $\aut{Y_L}$
is big; nevertheless, $\autp{\Hat Y_L}=0$.

\item[(ii)] Suppose now that $L\in O_2$. Then $L$ is tangent to $V$
at a unique point $v\in V$ and, by
Example~\ref{e.TangentialSeveri}, $L\supset Q\ni v$, where
$Q\subset V$ is a nonsingular $\frac n2$-dimensional quadric.
Furthermore, $\Sing{Y_L}=Y'_L=V_L$, $Y_L$ is irreducible, and
from Corollary~\ref{c.phs} and Proposition~\ref{p.equalAut} it
follows that $\autp{\Hat Y_L} \ne0$. It is clear that
$\Def{Y_L}>0$, and from the proof of Proposition
\ref{p.HyperplaneSection} it follows that $\pDef{Y_L}=0$. Thus
conditions~a) and c) of the Main Theorem are satisfied. This
example shows that condition~b) of the Main Theorem cannot be
lifted.

\item[(iii)] Suppose finally that $L\subset\P W$ is most special,
i.e. $L\in O_1$. Then $L$ is tangent to $V$ along a nonsingular
$\frac n2$-dimensional quadric $Q$ (the entry locus of a smooth
point $y\in Y$; cf.~Example~\ref{e.TangentialSeveri}), and $L$ is
tangent to $Y$ along (i.e. at the nonsingular points of) the
$(\frac n2+1)$-dimensional linear subspace $\<Q\,\>$. Thus
$Y'_L=\Sing{Y_L}=V_L\cup\<Q\,\>$, and it is easy to see that
$Y''_L=V_L$. We note that the cubic $Y_L$ is irreducible since
otherwise $Y_L$ would be a union of a hyperplane and a quadric in
$L$ and, contrary to the above formula, $Y'_L$ would be
degenerate in $L$. From Corollary~\ref{c.phs} and
Proposition~\ref{p.equalAut} it follows that $\autp{\Hat
Y_L}\ne0$, and so condition~c) of the Main Theorem is satisfied.

Next we study condition~a) of the Main Theorem. We will show
that, unlike in the cases~(i) and~(ii), in this case
$\pDef{Y_L}=1$ and condition~a) of the Main Theorem fails. Let
$u\in L$ be a general point, and put $\bP_u=\<\cP_x,u\>$, where
$x=\gamma(y)$ and $\cP_x=\gamma^{-1}(x)=\<Q\,\>$; thus $\bP_u\subset L$
is a linear subspace of dimension $N-n=n/2+2$. Since $Y$ is a
cubic and $\bP_u\subset L$, we have $\bP_u\cdot Y_L= 2\cP_x+
M_u$, where $M_u$ is a hyperplane in $\bP_u$. Put
$T(Q,V)=\cup_{\q\in Q}T_{V,\q}$. Then $T(Q,V)\subset L$ and from
\cite[Ch.~I, Theorem~1.4]{Z1} it follows that $Y_L=T(Q,V)$, so
that a general point $\y\in M_u$ is contained in $T_{V,q}$ for
some $q\in Q$, whence $ T_{V,q}\cap\bP_u=\overline{T_{V,q}\cap
(\bP_u \setminus\cP_x)}\subseteq\overline{Y_L\cap(\bP_u\setminus
\cP_x)}=M_u$. Furthermore, $M_u\cap\cP_x=T_{V,q}\cap\cP_x=
T_{Q,q}$, hence the point $q$ with the above property is unique,
$M_u\cap Q=K_q$ is a cone with vertex $q$, and $M_u=
\<\y,T_{Q,q}\>=T_{V,q}\cap\bP_u$. From the description given in
Examples~\ref{e.nonzeroprolongation} and \ref{e.TangentialSeveri}
it follows that, if $\pi_{\cP_x}:\P^N\dasharrow\P^n$ is the
projection with center $\cP_x$, then $\pi_{\cP_x}(L)=\P^{n-1}
\subset\P^n$ and $\pi_{\cP_x}(V_L)=B$, where $B\subset\P^{n-1}$
is a proper subvariety explicitly described in the above
examples. Thus, if $\pi_{\cP_x}(u)\notin B$, then
set-theoretically $V\cap\bP_u=Q$, i.e. the polar map on $\bP_u$
is defined by a linear system of quadrics whose base locus is
supported on $Q$. Since $\bP_u$ is tangent to $V$ at $q$, on the
scheme level we have $\bP_u\cdot Y'=Q\cup\bq$, where $\bq$ is a
double point supported on $q$. The linear system of {\it all\/}
quadrics in $\bP_u$ passing through $Q$ yields a birational map
$\tilde\phi$ of $\bP_u$ onto a nonsingular quadric $\bQ_u$,
$\dim{\bQ_u}=\dim{\bP_u}=n/2+2$ (this is an easy special case of
Example~\ref{e.nonzeroprolongation}). Furthermore, as we saw in
Examples~\ref{e.nonzeroprolongation} and
\ref{e.TangentialSeveri}, the map $\tilde\phi$ contracts the
hyperplane $\cP_x$ to a point $\tilde x\in\bQ_u$ and blows up the
quadric $Q$ to the cone $\bQ_u\cap T_{\bQ_u,\tilde x}$, and the
inverse map $\tilde\psi:\bQ_u \dasharrow\bP_u$ is given by
projecting from $\tilde x$. Let $\tq$ be the point of the line
$\tilde\psi^{-1}(q)$ corresponding to $\bold q$. Adding the fat
point $\bq$ to the fundamental scheme of $\tilde\phi$ is clearly
equivalent to taking a composition of $\tilde\phi$ with the
projection $\pi_{\tq}:\bQ_u\dasharrow\bP_{\bq}$ to an
$(n/2+2)$-dimensional linear subspace $\bP_{\bq}\subset
\P^{N\,{}^*}$ with center at $\tq$. Thus $\phi_{|\bP_u}=
\pi_{\tq}\circ\tilde\phi$, $\phi(M_u)\subset\bP_{\bq}\cap X$ is a
nonsingular $\frac n2$-dimensional quadric $Q_u\ni x$, and
$\pi_{\tq}(\tilde\psi^{-1}(K_q))=T_{Q_u,x}$ is a linear subspace
of $T_{X,x}$.

Let $\q\in Q$ be an arbitrary point. Then, by the above,
$\phi(\<u,\q\>)$ is a line in $\bP_{\bq}$ joining the point
$\phi(u)$ with a point of $T_{Q_u,x}$. If $\q\in Q$ is general,
then $m_q=\<u,\q\>\cap M_u\notin\cP_x$ and $\phi(\<u,\q\>)\subset
\<\phi(m_q),T_{Q_u,x}\>\subset S(Q_u,T_{Q_u,x})\subset
S(X,T_{X,x})$. If $m_q\in K_q$, then the line $\phi(\<u,\q\>)$
meets $Q_u$ only at a point of the cone $T_{Q_u,x}\cap Q_u$, and
if $m_q=q$, then the line $\phi(\<u,\q\>)$ passes through $x$.
Varying $u$, we get an $(n-1)$-dimensional family of linear
subspaces $\bP_u$ sweeping out $L$, so from our analysis it
follows that $\phi(L)\subset S(X,T_{X,x})$. Since $X\subset
\P^{N\,{}^*}$ is also a Severi variety, from the Terracini lemma
it follows that, for a general $\x\in X$, $\dim{S(X,T_{X,x})}=
\dim{\<T_{X,\x},T_{X,x}\>}=\dim{SX}=N-1=\dim{\phi(L)}$. Thus
$S(X,T_{X,x})\subset\P^{N\,{}^*}$ is an irreducible hypersurface,
$\phi(L)= S(X,T_{X,x})$, and $\phi(L)$ is a cone with vertex
$T_{X,x}$. Recall that, as it was observed in the proof of
Proposition~\ref{p.HyperplaneSection}, $Z_L=\phi_L(L)=
\pi_x(\phi(L))$, hence $Z_L$ is a hypersurface in
$\P^{N-1\,{}^*}$ and $\pDef{Y_L}=1$. Let $X^L=\pi_x(X)$. Under
the projection $\pi_x:X\dasharrow X^L$, the point $x\in X$ is
blown up to an $(n-1)$-dimensional linear subspace $\Lambda
\subset X^L$, where $\Lambda= \pi_x(T_{X,x})$ (cf.
Example~\ref{e.TangentialSeveri}). Thus $Z_L=S(X^L,\Lambda)$ is
the cone with vertex $\Lambda$ over the variety
$\pi_{\Lambda}(X^L)=\pi_{T_{X,x}}(X)\subset \P^{n/2+1\,{}^*}$.
Furthermore, $X$ is a Severi variety, hence, in the notations of
Example~\ref{e.nonzeroprolongation}, the projection
$\pi_{T_{X,x}}:\P^{N\,{}^*}\dasharrow\P^{n/2+1\,{}^*}$ is nothing
else but the natural projection $\P W\dasharrow\P K$, and so the
base $\pi_{T_{X,x}}(X)$ of the cones $\phi(L)$ and $Z_L$ is a
nonsingular quadric of dimension $n/2$. It is also worth
mentioning that $X^L$ has an intrinsic definition in terms of
$Y_L$, viz. $X^L=Y_L^*$ (a similar fact for {\it general\/} $L$
was mentioned in the proof of Lemma~\ref{l.red}; since $Y_L$ is
irreducible, the same proof works in our setup).

Finally we turn to condition~b) of the Main Theorem. For $i=1$,
$X^L$ is a rational cubic scroll, $Q=V_L$ is a conic,
$Y'_L=\<Q\,\>\supset Q$ is a plane, and condition~b) of the Main
Theorem is satisfied. For $i>1$, $Y'_L$ is reducible: for $i=2$
it has three components ($\P^3$ and two Segre threefolds
$\P^1\times\P^2$, all meeting along $Q$), and for $i=3,4$ there
are two components ($\P^{n/2+1}$ and $V_L$ meeting along $Q$).
Thus, for $i>1$, $Y'_L$ is singular, viz. $\Sing{Y'_L}=\Sing{V_L}
=Q$ and condition~b) of the Main Theorem fails.

Summing up, the polar image of any hyperplane from the orbit
$O_1$ is a cone with vertex $\P^{n-1}$ over a nonsingular $\frac
n2$-dimensional quadric. This example (particularly for $i=1$)
shows that condition~a) of the Main Theorem cannot be lifted.
\end{itemize}
\end{example}

\begin{remark}\label{l.r.} In Proposition~\ref{p.HyperplaneSection}
we showed that a {\it general\/} hyperplane section of a hypersurface
$Y$ with nonvanishing Hessian also has nonvanishing Hessian.
Example~\ref{e.SectionSecSeveri}\,(iii) shows that this need not be
so for {\it special\/} hyperplane sections: both Lemma~\ref{l.L} and
Proposition~\ref{p.HyperplaneSection} may fail and the Hessian of a
special hyperplane section may vanish. The arguments we used in
studying Example~\ref{e.SectionSecSeveri}\,(iii) actually show that
the same is true for a wide class of dual defective cubics $Y$ if we
consider their intersection with the tangent hyperplane at a general
point $y\in Y$.

For $i=1$, Example~\ref{e.SectionSecSeveri}\,(iii) is just the
classical example of Gordan and Noether (cf. a detailed description
in \cite[Example~4.6]{Z2}). For $i=2$,
Example~\ref{e.SectionSecSeveri}\,(iii) was studied by a different
method in \cite[Example~6]{GR}, and the cases $i=3,4$ were mentioned
in passing in \cite[Remark~5.1]{GR}.
\end{remark}

\medskip

Baohua Fu (bhfu@math.ac.cn)

\noindent MCM, AMSS, Chinese Academy of Sciences, 55 ZhongGuanCun
East Road, Beijing, 100190, China and School of Mathematical
Sciences, University of Chinese Academy of Sciences, Beijing, China \\

Yewon Jeong (ywjeong@amss.ac.cn)

\noindent Morningside Center of Mathematics, AMSS,  Chinese Academy
of Sciences, 55 ZhongGuanCun East Road, Beijing, 100190, China \\

Fyodor L. Zak (zak@cemi.rssi.ru)

\noindent CEMI RAS, Nakhimovski\v{i} av. 47, Moscow 117418, RUSSIA

\end{document}